\documentclass[12pt, reqno]{amsart}
\usepackage[body={6.3in,9.3in}]{geometry}
\usepackage{verbatim}
\usepackage[ansinew]{inputenc}

\usepackage{textcomp}
\usepackage{latexsym}
\usepackage[usenames]{color}
\usepackage{amsmath,amsfonts,amssymb,amsthm}
\usepackage{graphicx}
\usepackage{longtable}
\usepackage{bbm, cite}
\usepackage{algorithm, algorithmicx, algpseudocode}

\usepackage{setspace}
\usepackage{array}

\def\qed{\hfill{\raggedleft{\hbox{$\Box$}}} \smallskip}

\def\R{\mathbb{R}}

\def\N{\mathcal{N}}

\def\E{\mathbb{E}}
\def\P{\mathbb{P}}

\DeclareMathOperator{\Var}{Var}

\newcommand{\ds}{\displaystyle}

\theoremstyle{plain} \newtheorem{lem}{Lemma}
\theoremstyle{plain} \newtheorem{prop}[lem]{Proposition}
\theoremstyle{plain} \newtheorem{thm}[lem]{Theorem}
\theoremstyle{plain} \newtheorem{cor}[lem]{Corollary}
\theoremstyle{plain} 
\theoremstyle{plain} 
\theoremstyle{definition} 
\theoremstyle{definition}
\theoremstyle{definition} 
\theoremstyle{definition} 
\theoremstyle{definition}\newtheorem{ex}[lem]{Example}

\newlength\savedwidth

\def\1{\mathbf{1}}

\usepackage{tkz-graph}
\usepackage{tikz}
\usetikzlibrary{arrows}
\usepackage{mathtools, hyperref}
\usepackage{color}

\author{Francois Baccelli and Ngoc Mai Tran}
\address{Department of Mathematics, UT Austin, TX 78712, USA}
\thanks{This work was supported by an award from the Simons Foundation ($\#197982$ to The University of Texas at Austin)}

\DeclareMathOperator{\Tf}{\mathcal{T}f}

\title[Zeros of Random Tropical Polynomials]{Zeros of random tropical polynomials,\\
random polytopes and stick-breaking}

\begin{document}
\begin{abstract}
For $i = 0, 1, \ldots, n$, let $C_i$ be independent and identically distributed random variables with distribution $F$ with support $(0,\infty)$. The number of zeros of the random tropical polynomials $\Tf_n(x) = \min_{i=1,\ldots,n}(C_i + ix)$ is also the number of faces of the lower convex hull of the $n+1$ random points $(i,C_i)$ in $\R^2$. We show that this number, $Z_n$, satisfies a central limit theorem when $F$ has polynomial decay near $0$. Specifically, if $F$ near $0$ behaves like a $gamma(a,1)$ distribution for some $a > 0$, then $Z_n$ has the same asymptotics as the number of renewals on the interval $[0,\log(n)/a]$ of a renewal process with inter-arrival distribution $-\log(Beta(a,2))$. Our proof draws on connections between random partitions, renewal theory and random polytopes. In particular, we obtain generalizations and simple proofs of the central limit theorem for the number of vertices of the convex hull of $n$ uniform random points in a square. Our work leads to many open problems in stochastic tropical geometry, the study of functionals and intersections of random tropical varieties.
\end{abstract}

\maketitle

\section{Introduction}\label{sec:intro}

Consider the tropical min-plus algebra $(\R, \odot, \oplus)$, $a \odot b = a + b$, $a \oplus b = \min(a,b)$. A tropical polynomial $\Tf: \R \to \R$ of degree $n$ has the general form
\begin{equation}\label{eqn:tf}
\Tf(x) = \bigoplus_{i=0}^n(C_i\odot x^i) = \min_{i=0,\ldots,n} (C_i + ix),
\end{equation}
for coefficients $C_i \in \R$. The zeros of $\Tf$ are points in $\R$ where the minimum in~(\ref{eqn:tf}) is achieved at least twice. When the coefficients $C_i$'s are random, the zeros of $\Tf$ form a collection of random points in $\R$. A natural model of randomness is one where the $C_i$'s independent and identically distributed (i.i.d.) according to some distribution $F$, called the \emph{atom distribution} \cite{tao.vu}. In this paper, we derive the asymptotic distribution for the number of zeros as $n \to \infty$, under various atom distributions.

\begin{thm}\label{thm:main.rn} Let $F$ be a continuous distribution, supported on $(0,\infty)$. Assume $F(y) \sim Cy^a + o(y^a)$ as $y \to 0$ for some constants $C, a > 0$. Let $Z_n$ be the number of zeros of $\Tf$ in (\ref{eqn:tf}). Then as $n \to \infty$, 
\begin{equation}\label{eqn:main}
\frac{Z_n - \frac{2a+2}{2a+1}\log(n)}{\sqrt{\frac{2a(a+1)(2a^2+2a+1)}{(2a+1)^3}\log(n)}} \stackrel{d}{\to} \mathcal{N}(0,1).
\end{equation}
\end{thm}

The classical analogue of $Z_n$ the number of real zeros of a classical polynomial $ f(x) = \sum_{i=0}^nC_ix^i$ with random coefficients $C_i \in \R$ or $\mathbb{C}$. This problem has an extensive literature, ranging from works in the mid-twentieth century \cite{kac1, kac2,littlewood1938number} to a very recent paper by Tao and Vu \cite{tao.vu}, who proved a local universality phenomenon. Roughly speaking, this asserts that the asymptotic behavior of the zeros of $f$ as $n \to \infty$, appropriately normalized, should become independent of the choice of the atom distribution. Our main result is a version of this statement for tropical polynomials. 

Our setup provides a natural solution to counting zeros of polynomials over fields with valuations such as the $p$-adic numbers or Puiseux series, where methods developed for $\R$ and $\mathbb{C}$ do not easily apply. Such a polynomial $f$ comes with a tropicalized version $\Tf$. By the Fundamental Theorem of Tropical Algebraic Geometry, the zeros of $\Tf$ correspond to valuations of the zeros of $f$ (see, for example, \cite{bernd.trop}). In particular, $f$ and $\Tf$ have the same number of zeros. Studying combinatorial properties of classical varieties via tropicalization is the heart of tropical algebraic geometry. While zeros of random polynomials in fields with valuations have been studied \cite{evans}, to the best of our knowledge this is the first result in the tropical settings.

\subsection{Connections to random polytopes} 
%The proof is based on an extension of the results of Groeneboom \cite{groenboom88} on the number of vertices of the convex hull of uniformly distributed points in a square.
Let $\mathcal{C}(i,C_i)$ denote the lower convex hull of the point $(i,C_i)$. These are faces with support vectors of the form $(1,\alpha)$ for some $\alpha \in \R$. Let $|\mathcal{C}(i,C_i)|$ be the number of such faces. As we shall review in Lemma \ref{lem:legendre}, 
$$ Z_n = |\mathcal{C}(i,C_i)|. $$
This connects random tropical polynomials with random polytopes. More explicitly, suppose the atom distribution $F$ is $Uniform(0,1)$. If we replace $i = 0, 1, \ldots, n$ by $n+1$ i.i.d. uniform points $U_i$ on $(0,1)$, then $(U_i,C_i)$ are $n+1$ uniform points on $(0,1)^2$, and their convex hull is a random polytope. Statistics of such random polytopes have been studied extensively, see \cite{schneider} for a recent review. For the convex hull of $n$ uniform points in a square, Groeneboom \cite{groenboom88} derived the central limit theorem for its number of vertices. It follows from his proof that the number of lower faces $|\mathcal{C}(U_i,C_i)|$ satisfies
\begin{equation}\label{eqn:general.r}
\frac{\E(|\mathcal{C}(U_i,C_i)|) - \frac{4}{3}\log(n)}{\sqrt{\frac{20}{27}\log(n)}} \stackrel{d}{\to} \N(0,1),
\end{equation}
and this is precisely Theorem \ref{thm:main.rn} for the case $a=1$. In fact, our proof of Theorem \ref{thm:main.rn} is based on an extension of Groeneboom's results. 
In the more recent paper \cite{groenboom12}, Groeneboom re-derived his results in \cite{groenboom88} using a simpler argument with very similar ideas. The key idea of our proof, Lemma \ref{lem:key}, appeared as Corollary 2 in \cite{groenboom12}. However, he did not make the connection to stick-breaking or renewal theory explicit, nor generalize to the case of non-homogeneous PPP.  

\subsection{Connections to stick-breaking} 
For $C_i$ the $i$-th value of a random walk with exchangeable increments, the lower convex hull $\mathcal{C}(i,C_i)$ is also known as the \emph{greatest convex minorant} of the walk $(C_i, i = 0, 1, \ldots, n)$. Various authors have studied greatest convex minorants (or concave majorants) of random walks \cite{joshPitman}, Brownian motion \cite{rossPitman, groeneboom1983concave}, L\'{e}vy processes and other settings \cite{bertoin2000convex, pitman2012convex}. A classical result by Andersen \cite{andersen1954fluctuations} states that if almost surely no two subsets of the increments have the same arithmetic mean, then $|\mathcal{C}(i,C_i)|$ is distributed as the number of cycles in a uniformly distributed random permutation of the set $\{1, \ldots, n\}$. Its asymptotics in this case is
$$
\frac{|\mathcal{C}(i,C_i)| - \log(n)}{\sqrt{\log(n)}} \stackrel{d}{\to} \N(0,1).
$$
Paraphrased, this is a very strong \emph{universality result} on the zeros of tropical polynomials generated randomly under this settings. In fact, an even stronger result holds: the partition of $n$ generated by the lengths of the faces in $\mathcal{C}(i,C_i)$ has the same distribution as the partition of $n$ generated by the cycles of a uniform random permutation. As $n \to \infty$, the lengths of these cycles converge in distribution to the lengths of sticks obtained from a \emph{uniform stick-breaking process} (cf. Section \ref{sec:stickbreak}). 

In Section \ref{sec:stickbreak}, we show that when $F$ is $exponential(1)$, conditioned on $C_n = 0$, the partition of $n$ generated by the lengths of the faces in $\mathcal{C}(i,C_i)$ is a partially exchangeable partition. The limiting distribution of cycle lengths converge to the partition lengths of a $Beta(2,1)$ stick-breaking process. To our knowledge, the $Beta(2,1)$ scheme has not been considered in the literature. The classical two-parameter family contains the $Beta(1,\theta)$ stick-breaking scheme \cite[\S 3]{csp}, \cite{tamara}. A fundamental difference is exchangeability: in the $Beta(1,\theta)$ case, one obtains an exchangeable partition of $n$, where as our partition is only partially exchangeable. However, like the classical case, the $Beta(2,1)$, and in general, the $Beta(2,a)$ stick-breaking scheme for $a > 0$ enjoy the Polya's urn connection. In particular, it is a version of the Bernoulli sieve of Gnedin et. al. \cite{gnedin, gnedin2}. For $F = exponential(1)$, $a = 1$, and the number of cycles in a $Beta(2,1)$ partition is precisely the number zeros of $\Tf$ conditioned on $C_n = 0$. In this case, $G$ of Proposition \ref{thm:ppp} is just the Lebesgue measure, and we are back to the settings of Groeneboom \cite{groenboom88}. By conditioning on the minimum index of the $C_i$'s, one obtains an elementary stick-breaking proof of the results of Groeneboom \cite{groenboom88} (cf. Section \ref{sec:stickbreak}).

\subsection{Proof overview}
To prove (\ref{eqn:general.r}), Groeneboom showed that the points of $(U_i,C_i)$ can be replaced by the points in $(0,1)^2$ of a homogeneous Poisson point process (PPP) with rate $n$. He then studied their lower convex hull by a pure jump Markov process on the vertices, indexed by the slopes of their support vectors.

Our first step is a generalization of this result to a class of non-homogeneous PPP on~$\R_+^2$. More precisely, let $F$ be the distribution function in the hypothesis of Theorem~\ref{thm:main.rn}. Consider the PPP with intensity measure 
$n \lambda \times G$, where $\lambda$ is the Lebesgue measure and $G$ is the
measure on $\R_+$ such that, for $x \geq 0$,
$$G([0,x])= -\ln(1-F(x)). $$
The generalization in question is Proposition \ref{thm:ppp} below. We use the following definition:\\

{ \setlength{\leftskip}{1cm} \setlength{\rightskip}{1cm} \noindent 
{\em For a sequence of random variables $(X_n, n \geq 1)$
and a deterministic sequence $(b_n, n \geq 1)$, we say $X_n = \mathcal{O}_P(b_n)$
if for all sequences $(c_n, n \geq 1)$, $c_n \to \infty$ as $n \to \infty$,}
$$ \P(|X_n| \geq b_nc_n) \to 0 \mbox{ as } n \to \infty. $$
}

\begin{prop}\label{thm:ppp}
Let $\lambda$ be the Lebesgue measure, $F$ be the distribution in Theorem \ref{thm:main.rn}, and $\Phi_n$ be a Poisson point process on $\R_+^2$ with intensity measure $n\lambda \times G$. For points of $\Phi_n$ in $(0,1) \times (0,F^{-1}(1-e^{-1}))$, let $|\mathcal{C}(\Phi_n)|$ be the number of lower faces in their convex hull. Then
$$ |\mathcal{C}(\Phi_n)| - (J_n + J_n') = \mathcal{O}_P(1), $$
where $J_n$ and $J'_n$ are independent random variables, each distributed as the number of renewals on $[0,\log(n)/a]$ of a delayed renewal process with inter-arrival distribution $-\log(Beta(a,2))$. Consequently, as $n \to \infty$, 
$$ \frac{|\mathcal{C}(\Phi_n)| - \frac{2a+2}{2a+1}\log(n)}{\sqrt{\frac{2a(a+1)(2a^2+2a+1)}{(2a+1)^3}\log(n)}} \stackrel{d}{\to} \mathcal{N}(0,1).$$
\end{prop}
Theorem \ref{thm:main.rn} is a discrete version of this setup. Indeed, let $\lambda_n$ be a discrete measure on $\R$ which puts mass $1$ at every point $i/n$ for $i = 0, 1, \ldots$, and $0$ elsewhere. On each half line $\{i/n\} \times (0,\infty)$, run an independent PPP with intensity measure $G$. By design, the first jump of this process is distributed as $F$. But only the first point can possibly contribute to the lower convex hull. Thus the number of tropical zeros $Z_n$ equals $|\mathcal{C}(\widetilde \Phi_n)|$, where $\widetilde \Phi_n$
is a PPP with intensity measure $\lambda_n \times G$. A direct coupling between $\Phi_n$ and $\widetilde\Phi_n$ (cf. Proposition~\ref{prop:coupling}) proves that
$$ |\mathcal{C}(\Phi_n)| - |\mathcal{C}(\widetilde \Phi_n)| = \mathcal{O}_P(1), $$
and Theorem \ref{thm:main.rn} follows.

\subsection{Scope and organization} We hope to kindle interests in researchers from both stochastic and tropical geometry. However, this paper has a rather narrow scope of counting the number of zeros for a natural model of random tropical polynomials. This is perhaps the simplest linear functional of the simplest tropical variety. A review of both fields with rigorous definition of stochastic tropical geometry is better suited for subsequent work. 

Our paper is organized as follows. In Section \ref{sec:background} we prove the connection between zeros of $\Tf$ and the lower convex hull of the points $(i,C_i)$. While this is a special case of the well-known connection between regular subdivision of Newton polytopes and tropical varieties, we provide an explicit proof via Legendre transform for self-containment. Sections \ref{sec:stickbreak} to~\ref{sec:connections} treat the case $F = exponential(1)$ via two different proofs: discrete stick-breaking and Poisson coupling. These proofs and their interactions provide intuition for the proof of the main case. Section \ref{sec:main} proves Proposition \ref{thm:ppp}. Section \ref{sec:coupling} proves Theorem \ref{thm:main.rn} via a coupling argument. We summarize the paper and discuss open problems in Section \ref{sec:discuss}.

\subsection*{Notation} For a set of points $(x_i,y_i)$, let $\mathcal{C}(x_i,y_i)$ denote their lower convex hull. Let $|\mathcal{C}(x_i,y_i)|$ denote the number of faces.  For an underlying $\mathcal{C}(x_i,y_i)$, we list its vertices $(V_i, i \geq 0)$ in increasing $x$-coordinate. For a vertex $V$ of $\mathcal{C}(x_i,y_i)$, let $x(V)$ denote its $x$-coordinate, $y(V)$ its $y$-coordinate, and $i(V) \in \{0, \ldots, n\}$ its index. Identify the sequence $(x(V_j) - x(V_{j-1}), j = 1 , \ldots, |\mathcal{C}(x_i,y_i)|)$ with a partition of $[0,1]$ ordered by appearance, denoted $\Pi(\mathcal{C}(x_i,y_i))$. Similarly, identify the sequence $(i(V_j) - i(V_{j-1}), j = 1 , \ldots, |\mathcal{C}(x_i,y_i)|)$ with a partition of $n$ ordered by appearance, denoted $\Pi_n(\mathcal{C}(x_i,y_i))$.

We often split $\mathcal{C}(x_i,y_i)$ into the `lower left' $\mathcal{C}^+(x_i,y_i)$ and `lower right' $\mathcal{C}^-(x_i,y_i)$ convex hulls, the first consists of faces with support vectors $(1,\alpha)$ for $\alpha > 0$, and the second consists of those with $\alpha < 0$. Let $|\mathcal{C}^+(x_i,y_i)|$ and $|\mathcal{C}^-(x_i,y_i)|$ be the corresponding number of faces. For an underlying  $\mathcal{C}^+(x_i,y_i)$ or  $\mathcal{C}^-(x_i,y_i)$, let $(V_i^\downarrow, i \geq 0)$ be the vertices of listed in decreasing $y$-coordinate, $(V_i^\uparrow, i \geq 0)$ be the same set of vertices listed in increasing $y$-coordinate. For points in a fixed rectangle of some point process $\Phi$, we denote their lower convex hull by $\mathcal{C}(\Phi)$. Analogous quantities such as $\mathcal{C}^+(\Phi)$, $\Pi(\mathcal{C}(\Phi))$ follow the same naming convention. 

For $\alpha \in (0,\infty)$, let $L(\alpha)$ be the line orthogonal to the vector $(1,\alpha)$ and which supports $\mathcal{C}^+(\Phi)$. Let $L_x(\alpha)$ and $L_y(\alpha)$ be its $x$ and $y$-intercepts, respectively. Let $(x(\alpha),y(\alpha))$ be the vertex of $\mathcal{C}^+(\Phi)$ supported by $L(\alpha)$. If there are two or more such vertices, take the one with minimum $y$-coordinate. 

\section{Background}\label{sec:background}
We now derive the connection between the zeros of our tropical polynomial $\Tf$  in (\ref{eqn:tf}) and the lower convex hull $\mathcal{C}(i,C_i)$. This is a special case of a classical result in tropical algebraic geometry, see \cite[\S 2]{bernd.trop}, which has been rediscovered several times across different literature \cite{abg05, powerdiagram}. Write
$$ g(x) = -\Tf(x) = \max_{i=0,\ldots,n} (-C_i) + (-i)x. $$
Then $g$ is a convex, piecewise linear function. Its Legendre transform $\hat{g}$ is also convex and piecewise linear, given by
$$\hat{g}(w) =
\sup_x (wx -g(x)).$$
%\inf\{\sum_i \lambda_i C_i: \sum_i(-i)\lambda_i = w, \lambda_i \geq 0, \sum_i \lambda_i = 1\}. $$
% The zeros of $\Tf$ are its points of discontinuity. 
For $w \in [-n,0]$, $\hat{g}(w)$ is finite, and $-\hat{g}(w)$ is the $y$-intercept of the tangent to the graph of $g$ with slope $w$. Since $g$ is piecewise linear, the tangent line to $g$ only changes at its angular points (the points of discontinuity of its slope). We have \cite{powerdiagram, abg05}
$$\hat{g}(w) =
\inf\{\sum_i \lambda_i C_i: \sum_i(-i)\lambda_i = w, \lambda_i \geq 0, \sum_i \lambda_i = 1\}. $$
Thus, over $[-n,0]$, the graph of $\hat{g}$ is precisely $\mathcal{C}(-i,C_i)$, the lower faces of the convex hull of the set of points $\{(-i,C_i): i = 0, \ldots, n\}$. The angular points of $g$, which are the zeros of $\Tf$, are hence in bijective correspondence to the slopes of the faces of $\hat{g}$. There is also a clear bijection between the faces of $\mathcal{C}(-i,C_i)$ and those of $\mathcal{C}(i,C_i)$. We summarize these observations below.

\begin{lem}\label{lem:legendre} For $\Tf$ in (\ref{eqn:tf}), there is a bijection between the zeros of $\Tf$ and the faces of $\mathcal{C}(i,C_i)$. The multiplicity of a zero is the lattice length of the corresponding face. In particular, the number of zeros of $\Tf$, counting multiplicity, is the number of faces of $\mathcal{C}(i,C_i)$, counting their lattice lengths.
\end{lem}

In light of Lemma \ref{lem:legendre}, it may be more natural to work with the \emph{max-plus} tropical algebra, where the polynomials are convex. Indeed, the max-plus algebra found applications in many areas \cite{BCOQ, elsner, gaubert, nmt}. We chose to work with min-plus, following the convention of tropical algebraic geometry \cite{bernd.trop}. 

\begin{figure}[h]
\includegraphics[width = \textwidth]{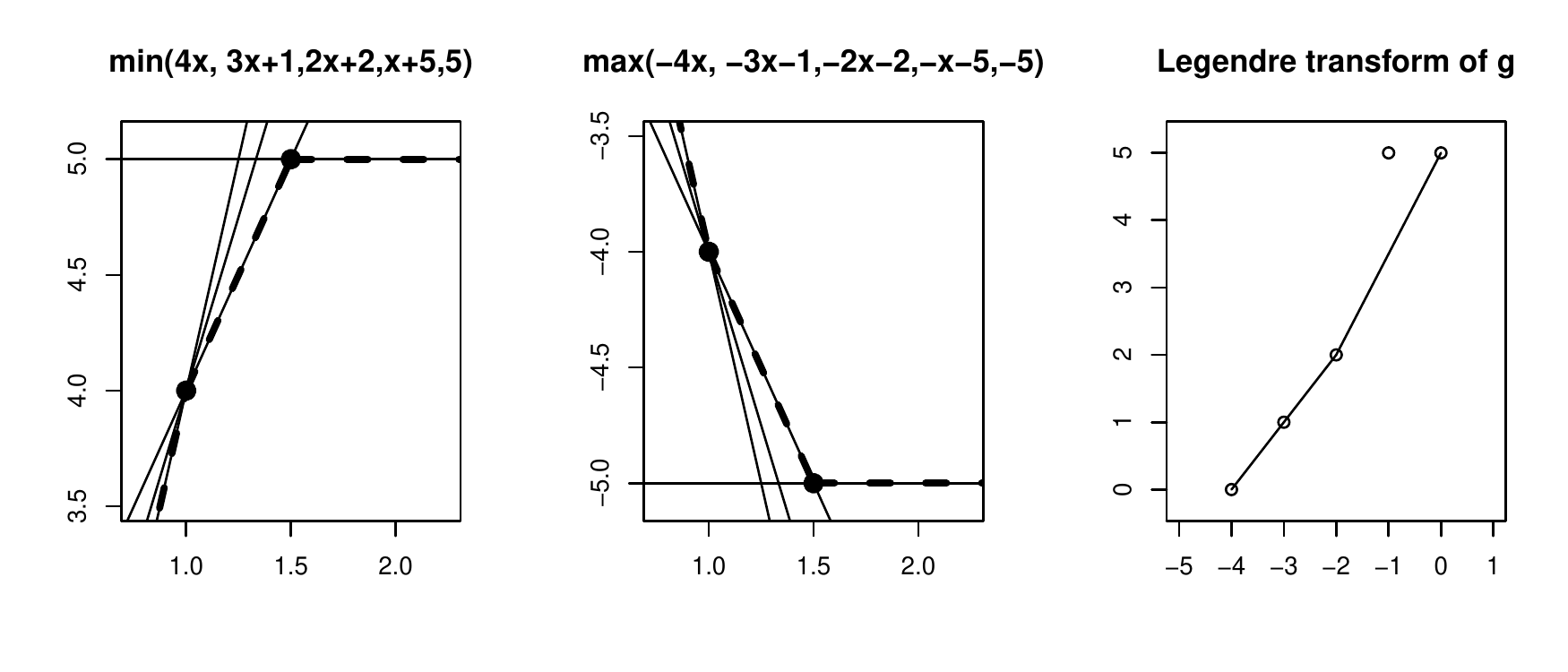}
\vskip-1.2cm
\caption{Graph of $\Tf$, $g$ and the Legendre transform of $g$.
The dotted thick lines are the actual graph of $\Tf$ and $g$, the continuous lines are the graph of the individual terms. \protect\label{fig:ex5}}
\end{figure}

\begin{ex}
Consider $\Tf(x) = 5\oplus5x\oplus2x^2\oplus1x^3\oplus x^4$. Figure \ref{fig:ex5} shows the graph of $\Tf$, $g$ and the Legendre transform of $g$. In this example, $\Tf$ has a double zero at $x = 1$ and another zero at $x = 3/2$. The Legendre transform of $g$ shows a face of lattice length 2 with slope 1, corresponding to the double zero at $x = 1$, and a face of lattice length 1 with slope $3/2$, corresponding to the zero at $x = 3/2$. The projection of the Legendre transform of $g$ onto the $x$-axis creates a subdivision of the line $-[0,4]$, which we identify with the line $[0,4]$. In this case the partition $\Pi_4(\mathcal{C}(i,C_i))$ is $(2,1,1)$. 
\end{ex}
 
\section{Stick-breaking proof for the exponential case}\label{sec:stickbreak}
We now derive a proof of Theorem \ref{thm:main.rn} for the case $F = exponential(1)$ by considering the partition $\Pi_n(\mathcal{C}(i,C_i))$. Computations in this case are significantly simpler, and they give insights into the proof of the general case. Furthermore, this setup connects our results with those in the literature, as discussed in Section \ref{sec:connections}. Here $a = 1$, and (\ref{eqn:main}) reads
$$ \frac{Z_n - \frac{4}{3}\log(n)}{\sqrt{20/27 \log(n)}} \stackrel{d}{\rightarrow} \N(0,1). $$

Recall that $V_0^\uparrow$ is the vertex of $\mathcal{C}(i,C_i)$ with minimum $C_i$. Since $F$ is continuous, its index $i(V_0^\uparrow)$ is a.s. unique. It is uniformly distributed on $\{0, \ldots, n\}$. Conditioned on $i(V_0^\uparrow) = k$, the random variables $C_i - C_k$ are i.i.d. $exponential(1)$. The partitions $\Pi_k(\mathcal{C}^+(i,C_i))$ and $\Pi_{n-k}(\mathcal{C}^-(i,C_i))$ are distributed as independent partitions $\Pi_k(\mathcal{C}(i,C_i))$ and $\Pi_{n-k}(\mathcal{C}(i,C_i))$ conditioned on the last $C_i$ to be zero. Thus it is sufficient to derive the distribution of $Z_n$ conditioned on the event $\{C_0 = 0\}$, in which case $\mathcal{C}(i,C_i) = \mathcal{C}^-(i,C_i)$.

\begin{prop}\label{prop:key}
For $n = 1, 2, \ldots$, let $p_n$ be the joint distribution function of $\Pi_n(\mathcal{C}^-(i,C_i))$. Then for all $m\in \{1, 2, \ldots, n\}$
and $x_1,\ldots,x_m\in \mathbb{N}$ with $\sum_{i=1}^m x_i=n$,
\begin{equation}\label{eqn:pn}
p_n(x_1, \ldots, x_m) = \prod_{i=1}^m\frac{x_i}{\binom{n-s_i+1}{2}},
\end{equation}
where $s_i = \sum_{j<i}x_j$, $s_1 = 0$. 
\end{prop}
\begin{proof}
In this case, for any vertex $V$, $x(V) = i(V)$. 
Thus the partition $\Pi_n(\mathcal{C}^-(i,C_i))$ identifies with the sequence $(X_j, j = 1, \ldots, |\mathcal{C}^-(i,C_i)|)$ where $X_j = x(V_j)- x(V_{j-1})$. Equation (\ref{eqn:pn}) is the probability of the event $\{X_j = x_j, j = 1, \ldots, m = |\mathcal{C}^-(i,C_i)|\}$. 

Note that $x(V^\uparrow_1)$ is distributed as the argmin of $\{C_i/i, i = 1, \ldots, n\}$ for i.i.d. $exponential(1)$ $C_i$'s. Thus 
\begin{equation}\label{eqn:xi}
 \P(X_1 = x_1) = \frac{x_1}{\binom{n+1}{2}}.
\end{equation}
Conditioned on $x(V_1) = k$, for $i = k+1, \ldots, n$, $C_i \stackrel{d}{=} \frac{i}{k}C_k + \epsilon_i$ where $\epsilon_i$ are i.i.d. exponential(1). Thus, $x(V^\uparrow_2)$ is distributed as the argmin of 
$$\frac{C_i-C_k}{i-k} = \frac{C_k}{k} + \frac{\epsilon_i}{i-k} $$
$i = k+1, \ldots, n$, which is the integer that achieves the minimum of $ \ds \frac{\epsilon_i}{i-k}$. So conditioned on $x(V^\uparrow_1) = k$, $x(V^\uparrow_2)$ is distributed as
$$ \P(x(V_2) = k') = \frac{k'-k}{\binom{n-k+1}{2}}.$$
Now, $X_2 = x(V_2) - x(V_1)$. Thus conditioned on $X_1 = x_1$, $X_2$ is distributed as
$$ \P(X_2 = x_2) = \frac{x_2}{\binom{n-k+1}{2}} = \frac{x_2}{\binom{n-s_2+1}{2}}.$$
Repeating the argument for $j = 3, 4, \ldots$ completes the proof.
\end{proof}
\begin{comment}
\begin{center}
\begin{figure}
\includegraphics[width=0.5\textwidth]{stickbreak5.png}
\caption{Recursion step demonstrated on $\mathcal{C}^+(i,C_i)$. Conditioned on $V^\uparrow_0$ and $V^\uparrow_1$, $y(V^\uparrow_i)$ of the remaining points lie above the line joining $V^\uparrow_0$ and $V^\uparrow_1$, and their vertical distance from the line is distributed as i.i.d. $exponential(1)$ random variables. The situation for $\mathcal{C}^-(i,C_i)$ mirrors this image. Thus, it is sufficient to assume $C_0 = 0$.}
\end{figure}
\end{center}
\end{comment}

If we just keep track of the number of components of $\Pi_n(\mathcal{C}^-(i,C_i))$, then the recursion argument of Proposition \ref{prop:key} translates into the following recurrence relation, which appears in \cite[Theorem 1]{buchta}. We expand on this connection in Section \ref{sec:connections}.

\begin{cor}
Let $p_k^n$ be the probability that the partition $\Pi_n(\mathcal{C}^-(i,C_i))$ has $k$ components. Then 
\begin{equation}\label{eqn:pkn}
p_k^n = \frac{1}{\binom{n+1}{2}}\sum_{j=k-1}^{n-1}(n-j)p_{k-1}^{j}.
\end{equation}
\end{cor}

The recursion in Proposition \ref{prop:key} has a Polya's urn scheme interpretation. Recall the basic setup: start with $w$ white balls and $b$ black balls. At each step, remove a random ball from the urn and replace with two balls of the same color. Repeat this process $n$ times, thus adding in total $n$ balls to the urn. Let $W_n$ denote the number of white balls newly added. Then for $x \in \{0, 1, \ldots, n\}$, 
$$ \P(W_n = x) = \binom{n}{x}\frac{(w+x-1)!(b+n-x-1)!(w+b-1)!}{(w-1)!(b-1)!(w+b+n-1)!} $$
In particular, for $w = 2, b = 1$, then 
\begin{equation}\label{eqn:polya}
\P(W_n = x) = \frac{x+1}{\binom{n+2}{2}}.
\end{equation}
Thus, for each $n$, $X_1 - 1$ is distributed as the number of newly added white balls after $n-1$ steps of the Polya's urn scheme starting with \emph{two} white balls and one black ball. Conditioned on $X_1 = x_1$, $X_2 - 1$ is distributed as the number of newly added white balls after $n-x_1-1$ steps of the same urn scheme, and so on. From the sequential description of Polya's urn, the sequence of partition functions $(p_n, n \geq 1)$ can be generated via the following variation of the \emph{Chinese restaurant process} (CRP). 
Introduced by Dubins and Pitman \cite[\S 3]{csp}, this is a model for consistent random permutations. Start with an initially empty restaurant with an unlimited number of tables numbered $1, 2, \ldots$, each capable of seating an unlimited number of customers. Customers numbered $1, 2, \ldots$ arrive one by one and choose their seats according to the following rules.

\begin{algorithm}
\caption{$Beta(2,1)$ Chinese Restaurant Process}\label{alg:exact}
\begin{algorithmic}[1]
	\State Customer $1$ sits at table $1$.
	\State Suppose after $n$ customers, tables $1, 2, \ldots k$ are occupied with $x_1, \ldots, x_m$ customers, $\sum_ix_i = n$. The $(n+1)$-st customer joins according to the following rules: 
	\begin{itemize}
		\item He joins table $1$ with probability $\ds \frac{x_1 + 1}{n+2} $
		\item Sequentially for $i = 2, \ldots, m$, conditioned on not joining the previous tables $1 \leq j \leq i-1$, he joins table $i$ with probability $\ds \frac{x_i+1}{n-s_j+2}$, where $s_j = \sum_{j<i}x_i$.
		\item Conditioned on not joining any of the previous tables, he forms a new table.
		\end{itemize}
\end{algorithmic}
\end{algorithm}

\begin{ex}[n=3] The following tree computes the distribution of the random partition of $3$ customers. 
\begin{figure}[h]
\includegraphics[width=1.2\textwidth]{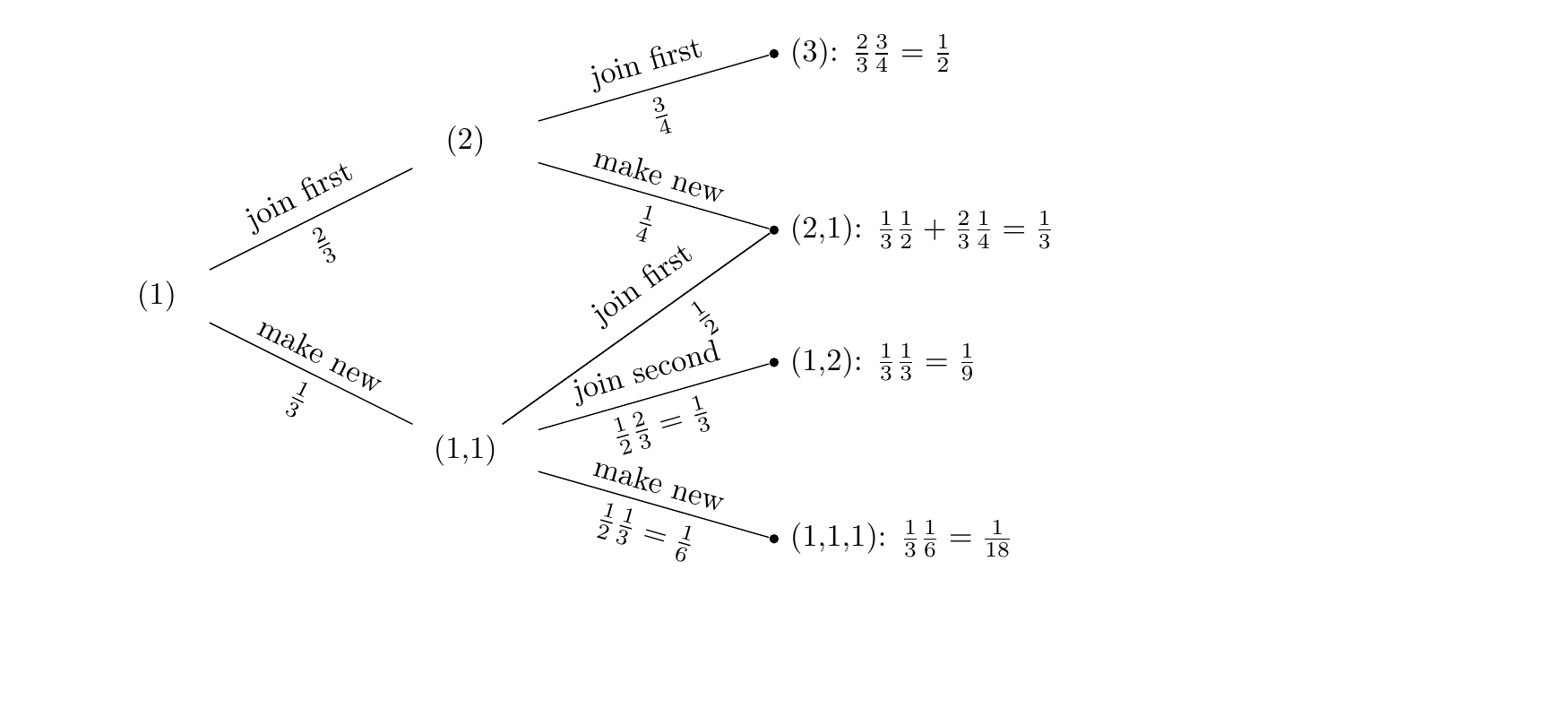}
\vskip-1cm
\end{figure}
\end{ex}
\begin{lem}\label{lem:crp.equals.pin} The partition of $n$ generated by the above CRP equals $\Pi_n(\mathcal{C}^-(i,C_i))$. In other words, for $n = 1, 2, \ldots$, the probability that tables $1, 2, \ldots, m$ have sizes $x_1, \ldots, x_m$ is exactly $p_n(x_1, \ldots, x_m)$ in (\ref{eqn:pn}).
\end{lem}
\begin{proof}
The claim is evident from the Polya's urn construction. Let us keep track of $N_1$, the size of the first table in the CRP. After $n$ steps, $N_1-1$ is distributed as the number of newly added white balls in the $(2,1)$-Polya's urn scheme after $n-1$ steps. Thus, $N_1 \stackrel{d}{=} X_1$. Similarly, let $N_2$ be the size of the second table. By construction, $N_2$ is distributed as the number of white balls in the $(2,1)$-Polya's urn scheme after $n-N_1-1$ steps. Therefore, $N_2 \stackrel{d}{=} X_2$. Repeating this argument proves the claim.
\end{proof}

As $n \to \infty$, the sequence of relative frequencies $(X_i/n, i \geq 1)$ converges in distribution to the continuous \emph{$Beta(2,1)$ stick-breaking sequence}
\begin{equation}\label{eqn:pi}
(P_1, P_2, \ldots) = (B_1, \bar{B}_1B_2, \bar{B}_1\bar{B}_2B_3,\ldots),
\end{equation}
where $B_i$ are i.i.d. $Beta(2,1)$ random variables, and $\bar{B}_i = 1 - B_i$ for $i \geq 1$. This sequence defines a distribution on the relative frequencies of a random partition $\Pi$ of $\mathbb{N}$. Let $p: \mathbb{N}^\ast = \bigcup_{k=1}^\infty \mathbb{N}^k \to [0,1]$ be the partition function of $\Pi$. 
\begin{lem}
Let $\Pi'_n$ be the restriction of $\Pi$ to $[n]$. Then $\Pi_n(\mathcal{C}^-(i,C_i)) \stackrel{d}{=} \Pi_n'$.
\end{lem} 
\begin{proof}
By \cite{pitmanPPF}, the probability that $\Pi'_n$ equals any specific partition $(x_1, \ldots, x_m)$ of $[n]$, in order of appearance, is given by  
\begin{equation}\label{eqn:pn.pi}
 p_n(x_1, \ldots, x_m) = \mathbb{E}\left(\prod_{i=1}^mP_i^{x_i-1}\prod_{i=1}^{m-1}(1-\sum_{j=1}^iP_j)\right).
\end{equation}
By a direct computation, we find that (\ref{eqn:pn.pi}) equals (\ref{eqn:pn}). This proves the lemma.
\end{proof}
We call $\Pi_n(\mathcal{C}^-(i,C_i))$ the \emph{discrete $Beta(2,1)$ stick-breaking scheme}. It follows from \cite[Theorem 6]{pitmanPPF} that there is another CRP representation for $\Pi_n(\mathcal{C}^-(i,C_i))$, this time conditioned on the limiting sequence $(P_1, P_2, \ldots)$. Specifically, given $(P_1, P_2, \ldots)$, and given that the partition $\Pi_n(\mathcal{C}^-(i,C_i))$ has sizes $(x_1, \ldots, x_m)$, $\Pi_{n+1}(\mathcal{C}^-(i,C_i))$ is an extension of $\Pi_n(\mathcal{C}^-(i,C_i))$ in which the $(n+1)$-st customer does the following:
\begin{itemize}
	\item Joins table $i$ with probability $P_i$, $1 \leq i \leq m$
	\item Joins a new table with probability $1 - \sum_{j=1}^mP_j$.
\end{itemize}
In other words, conditioned on $(P_i, i \geq 1)$, for fixed $n \geq 1$, $X_1 \stackrel{d}{=} Binomial(n, P_1)$. Conditioned on $X_1$, $X_2 \stackrel{d}{=} Binomial(n-X_1, P_2)$, and so on.  When $P_i = \prod_{j<i}\bar{B}_jB_i$ for i.i.d. random variables $B_j \sim B$ as in our case,  this is also known as the \emph{Bernoulli sieve} model, a recursive allocation of $n$ balls in infinitely many boxes $j = 1, 2, \ldots$. Here $X_j$ is the number of balls in the $j$-th box (if a box $j$ is not discovered then $X_j = 0$). Gnedin \cite{gnedin} and Gnedin et. al. \cite{gnedin2} studied the various functionals of this model, including $|\Pi_n(\mathcal{C}^-(i,C_i))|$, the number of boxes occupied by at least one ball. Through methods from renewal theory, they showed that $|\Pi_n(\mathcal{C}^-(i,C_i))|$ has the same asymptotics as the number of renewals on the interval $[0, \log(n)]$ of a renewal process whose inter-arrival time is distributed as $-\log(B)$. Specifically, for $\mu = \E(-\log(B))$ and $\sigma^2 = Var(-\log(B))$,
\begin{equation}\label{eqn:gnedin}
\frac{|\Pi_n(\mathcal{C}^-(i,C_i))| - \mu^{-1}\log(n)}{\sqrt{\sigma^2\mu^{-3}\log(n)}} \stackrel{d}{\rightarrow} \N(0,1).
\end{equation}
In our case, $B \stackrel{d}{=} Beta(2,1)$, thus $\mu = \frac{3}{2}, \sigma^2 = \frac{5}{4}$. Substituting into (\ref{eqn:gnedin}) yields
\begin{equation}\label{eqn:pi.asymp}
\frac{|\Pi_n(\mathcal{C}^-(i,C_i))| - \frac{2}{3}\log(n)}{\sqrt{\frac{10}{27}\log(n)}} \stackrel{d}{\rightarrow} \N(0,1).
\end{equation}
We can now complete the proof of Theorem \ref{thm:main.rn} for the case $F = exponential(1)$. 
\begin{proof}[Proof of Theorem \ref{thm:main.rn} when $F$ is $exponential(1)$]
As previously argued, we have
$$Z_n = |\Pi_{U_n}(\mathcal{C}^-(i,C_i))| + |\Pi'_{n-U_n}(\mathcal{C}^-(i,C_i))| $$
where $U_n$ is the discrete uniform distribution on $\{0, 1, \ldots, n\}$, and conditioned on $U_n = u$, $|\Pi_u(\mathcal{C}^-(i,C_i))|$ and $|\Pi'_{n-u}(\mathcal{C}^-(i,C_i))|$ are independent, with asymptotics given in (\ref{eqn:pi.asymp}). Thus $Z_n$, appropriately scaled, also converges to the standard normal distribution. It remains to find $\E(Z_n)$ and $\Var(Z_n)$. Couple $U_n$ with a continuous uniform $U \stackrel{d}{=} Uniform(0,1)$ in the obvious way. Then $|U_n/n - U| < \frac{1}{n}$ a.s. Therefore,
\begin{align*} 
\E(Z_n|U=u) &= \frac{2}{3}\left(\log(nu) + \log(n(1-u))\right) + O(1) \\
 &= \frac{2}{3}\log(n) + \frac{2}{3}(\log(u) + \log(1-u)) + O(1).
 \end{align*}
Therefore $\E(Z_n) = \frac{2}{3}\log(n) + O(1).$ Similarly,
\begin{align*}
\Var(Z_n|U=u) 
&= \frac{10}{27}\left(\log(nu) + \log(n(1-u))\right) + O(1) \\
&= \frac{20}{27}\log(n) + \frac{10}{27}(\log(u) + \log(1-u)) + O(1).
\end{align*}
By the above calculation, $\E(\Var(Z_n|U)) = \frac{20}{27}\log(n) + O(1)$, and
$$\Var(\E(Z_n|U)) = \Var(\frac{2}{3}(\log(u) + \log(1-u))) = O(1).$$
Thus, $\Var(Z_n) = \frac{20}{27}\log(n) + O(1)$. The central limit theorem for $Z_n$ follows by the Continuous Mapping Theorem (see, for example, \cite[\S 2]{durrett}).
\end{proof}

For any $a > 0$, we can define a discrete $Beta(2,a)$ stick-breaking process. One could ask whether there exists some $F$ such that $\Pi_n$ a partition from such a process. In the hindsight of the proof of Proposition \ref{thm:ppp}, such a process could arise when $F$ is the $gamma(a,b)$ distribution for some $b > 0$. However, 
Proposition \ref{prop:key} relies on the lack of memory of exponentials to obtain the recursion. We do not know how to generalize to other distributions in the gamma family. 

\section{Poisson proof for the exponential case}\label{sec:simple.poisson}

We now prove Proposition \ref{thm:ppp} when $F$ is $exponential(1)$. Here $G$ is the Lebesgue measure on $(0,\infty)$, $\Phi_n$ is a homogeneous Poisson point process on the first quadrant with rate $n$, and the rectangle in consideration is the square $(0,1)^2$. As discussed in the introduction, in this case, Proposition \ref{thm:ppp} is a version of Groeneboom's result \cite{groenboom88}. 

We now give an elementary proof of this result using stick-breaking. The key idea of this proof, Lemma \ref{lem:key} below, appeared as Corollary 2 in a later paper of Groeneboom \cite{groenboom12}. However, he did not make the connection to stick-breaking and renewal theorems explicit. Our proof clarifies the connection between the discrete and continuous setup, a key idea behind our main theorems. We discuss this in Section \ref{sec:connections}, together with connections to the results of Buchta \cite{buchta, buchta2} and Groeneboom \cite{groenboom88, groenboom12}. 

We first consider the lower convex hull $\mathcal{C}^+(\Phi_n)$ only. At the end of Section \ref{sec:intro} we introduced the process $(x(\alpha),y(\alpha)), \alpha \in (0,\infty)$ of vertices of $\mathcal{C}^+(\Phi_n)$ indexed by the slope of their support vectors. Note that $(y(\alpha), \alpha \in (0, \infty))$ is a pure-jump process indexed by increasing $\alpha$. The sequence of vertices of this process is precisely the sequence $(V_i^\downarrow, i \geq 0)$ of vertices of $\mathcal{C}^+(\Phi_n)$ ordered by decreasing $y$-values. The number of jumps is precisely $|\mathcal{C}^+(\Phi_n)|$. Groeneboom \cite{groenboom88} proved the following.

\begin{lem}[\cite{groenboom88}]
The process $((x(\alpha),y(\alpha)), \alpha \in (0,\infty)$ is a pure jump Markov process indexed by $\alpha$. Its number of jumps on $(0,\infty)$, $|\mathcal{C}^+(\Phi_n)|$, is a.s. finite.
\end{lem}

The infinitesimal generator for the process $(x(\cdot),y(\cdot))$ is rather complicated, see \cite{groenboom88}. Remarkably, the transition rule in the $y$-coordinate is very simple. Groeneboom discovered this fact in his more recent paper \cite{groenboom12}. 

\begin{lem}[\cite{groenboom88}]\label{lem:key}
Let $(B_i, i \geq 1)$ be an infinite sequence of i.i.d. $Beta(1,2)$random variables, independent of $y(V_0^\downarrow)$. The transition rule in $y(\cdot)$ is as follows. Given that $V^\downarrow_i = (x^\ast,y^\ast)$ for some $i \geq 0$, the process $y(\cdot)$ either terminates with probability $\exp(-n\frac{(1-x^\ast)y^\ast}2)$, in which case $i = |\mathcal{C}^+(\Phi_n)|$, or it jumps to a random point $y(V^\downarrow_{i+1})$, $0 < y(V^\downarrow_{i+1}) < y^\ast$, where
$$ y(V^\downarrow_{i+1}) \stackrel{d}{=} y^\ast B_{i+1}. $$
In other words, $y(V^\downarrow_0) \stackrel{d}{=} Uniform(0,1)$ and for all $i \le |\mathcal{C}^+(\Phi_n)|$, $ y(V_i) = y(V_0)\prod_{j=1}^iB_j.$
\end{lem}
\begin{proof} The proof is by induction on $i$. Note that $y(V^\downarrow_0)$ is distributed as a $Uniform(0,1)$. Almost surely, $y(V^\downarrow_0) = y(\alpha_0)$ for some $\alpha_0 > 0$. As $\alpha$ increases, the event that $y(\cdot)$ jumps from $y(V^\downarrow_0)$ to $y(V^\downarrow_1)$ in $(\alpha, \alpha+d\alpha)$ is the event that there is at least one point in the shaded triangle in Figure \ref{fig:small.triangle}.
\begin{center}
\begin{figure}[h]
\includegraphics{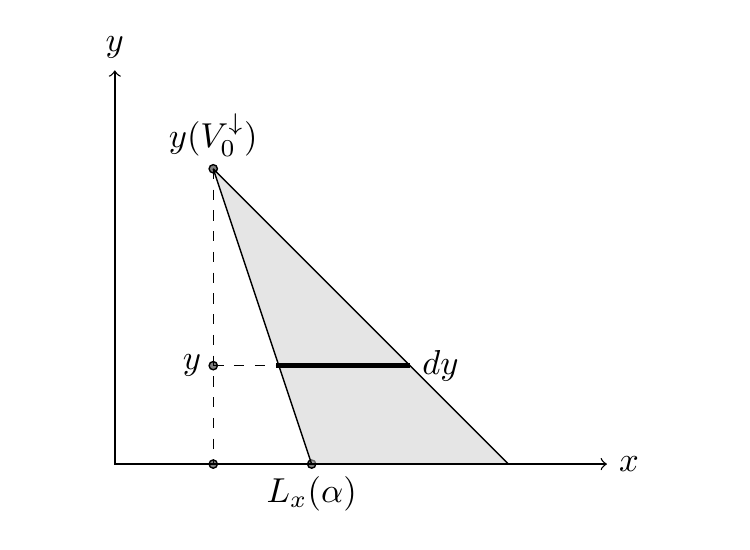}
\vskip-0.5cm
\caption{The process $y(\cdot)$ has a jump in $d\alpha$ if there is at least one point in the triangle. Conditioned on its existence, this point is uniformly distributed. \protect\label{fig:small.triangle}}
\end{figure}
\end{center}
Conditioned on $V^\downarrow_1$ being in this triangle, it is uniformly distributed. In particular, $\P(y(V^\downarrow_1)\in dy)$ is proportional to the length of the line highlighted, which is exactly proportional to $y(V^\downarrow_0) - y$. Note that this is independent of $\alpha$, as long as $\alpha > 0$. Thus $B_1 \stackrel{d}{=} Beta(1,2)$, and is independent of $y(V^\downarrow_0)$. Repeating this argument shows that $B_i$ as defined is distributed as $Beta(2,1)$, independent of $y(V^\downarrow_{i-1})$ and thus of other $B_j$'s, as long as $\alpha > 0$. The event of termination at some point $(x^\ast,y^\ast)$ is exactly when the point process $\Phi_n$ has no points in the rectangle with vertices $(x^\ast,y^\ast)$, $(1,y^\ast)$, $(0,y^\ast)$ and $(1,0)$. The probability of this event is $exp(-n(1-x^\ast)y^\ast/2)$. 
\end{proof}

Lemma \ref{lem:key} explains the appearance of $Beta(1,2)$: $y(V^\downarrow_0) - y(V^\downarrow_1)$ is a \emph{size-biased} pick from the $Uniform(0,y(V^\downarrow_0))$ distribution. On the phenomenon of size-bias in random combinatorial structures, see \cite{arratia, pitman.tran}). In fact, the discrete $Beta(1,2)$ appearing in Proposition \ref{prop:key} is also a size-bias phenomenon. We make the connection between Lemma \ref{lem:key} and Proposition \ref{prop:key} explicit in Section \ref{sec:connections}.

Excluding the stopping rule, $y(\cdot)$ is a discrete time Markov chain with some internal filtration. At each step, it terminates with some stopping rule that depends on the bigger filtration generated by the Markov process $(x(\cdot),y(\cdot))$. To avoid dealing with the bigger process, we couple $y(\cdot)$ with the Markov chain $(y'_i, i \geq 0)$, by taking $ y'_0 =y(V^\downarrow_0)$ and $y'_i = y'_0\prod_{j=1}^iB_j$, for the same sequence $(B_i, i \geq 1)$ of i.i.d. $Beta(2,1)$ random variables appeared in Lemma \ref{lem:key}. For any $t > 0$, let $J'(t)$ be the first time $i$ when $y'_i < e^{-t}$. Clearly $J'(t)$ is a stopping time for the Markov chain $(y'_i, i \geq 0)$. In fact, it is the number of renewals on $[0, t]$ of a delayed renewal process with i.i.d inter-arrival distribution $-\log(Beta(1,2))$, and first point distributed as $-\log(Uniform(0,1))$. As $t \to \infty$, by \cite[Theorem 2.5.1]{gut},
\begin{equation}\label{eqn:key.renewal}
\frac{J'(t)}{t} \stackrel{a.s.}{\to} \frac{1}{\mu},
\end{equation}
where $\mu = \E(-\log(Beta(1,2))) = \frac{3}{2}$.

The point of the coupling is that $J'(\log(n))$ (defined w.r.t. $(y'_i,i\ge 0))$ and $|\mathcal{C}^+(\Phi_n)|$ (defined w.r.t. the initial process) are close with high probability. Indeed, these two quantities are closed as long as $y(V_0^\uparrow) = y\left(V_{|\mathcal{C}^+(\Phi_n)|}^\downarrow\right)$ is close to $\frac{1}{n}$.

\begin{lem}\label{lem:approx}
Let $\delta(n) = |\mathcal{C}^+(\Phi_n)| - J'(\log(n))$. Then $\delta(n) = \mathcal{O}_P(1)$.
\end{lem}
\begin{proof}
We have to prove that for all
functions $\ell$ growing to infinity arbitrarily slowly,
$$ \P \left\{ J'(\log(n)) -\ell(n)\le |\mathcal{C}^+(\Phi_n)| \le  J'(\log(n)) +\ell(n) \right\} \to 1,$$ 
when $n\to \infty$. By~(\ref{eqn:key.renewal}) and the strong law of large numbers,
it is sufficient to show that for all $\ell$ as above,
\begin{equation}\label{eqn:bound.ymn}
\P\left\{ \log(n) -\ell(n)\le -\log(y(V^\uparrow_0))
\le  \log(n) +\ell(n)\right\} \to 1,
\end{equation}
when $n\to \infty$. 
We now prove (\ref{eqn:bound.ymn}) by using
the fact that $y(V^\uparrow_0)$ is the minimum of the
$y$-coordinates of the points of $\Phi_n$ in $(0,1)^2$. For this, we divide the region $(0,1) \times (0,\infty)$ into $n$ vertical strips $(i/n,(i+1)/n) \times (0,\infty)$ for $i = 0, \ldots, n-1)$. In each strip, the $y$-coordinate of the points of $\Phi_n$ form a homogeneous Poisson point process on $(0,\infty)$. In particular, in the $i$-th strip, the first $y$ jump, denoted by $\epsilon_i$, is distributed as a standard exponential. Let $Y' := \min(\epsilon_1, \ldots, \epsilon_n) \stackrel{d}{=} exponential(n)$. Then, conditioned on the event $Y' < 1$, which happens with probability $1 - e^{-n}$, $Y' = y(V^\uparrow_0)$. Now, 
$$\P(\frac{1}{n\ell(n)} \leq Y' \leq \frac{\ell(n)}{n}) = e^{-1/\ell(n)} - e^{-\ell(n)}.$$ 
Thus the probability of the event in (\ref{eqn:bound.ymn}) is at least
$$(1 - e^{-n})(e^{-1/\ell(n)} - e^{-\ell(n)}),$$
which concludes the proof.
\end{proof}
\begin{proof}[Proof of Proposition \ref{thm:ppp} when $F$ is exponential]
Write
$$ |\mathcal{C}(\Phi_n)| = |\mathcal{C}^+(\Phi_n)| + |\mathcal{C}^-(\Phi_n)| = J'(\log(n)) + J''(\log(n)) + \delta^+(n) + \delta^-(n). $$
Here $J'$ and $J''$ are stopping times of independent versions of the Markov chain $y'$ previously defined, and the terms $\delta^+(n), \delta^-(n)$ are error terms. Conditioned on the minimum $y$-coordinate $y(V_0^\uparrow)$, the pair $|\mathcal{C}^+(\Phi_n)|$ and $|\mathcal{C}^-(\Phi_n)|$, hence the pair $\delta^+(n)$ and $\delta^-(n)$, are independent. By (\ref{eqn:bound.ymn}), $\delta^+(n)$ and $\delta^-(n)$ are both $\mathcal{O}_P(1)$, therefore so is their sum. So
$$ |\mathcal{C}(\Phi_n)| = J'(\log(n)) + J''(\log(n)) + \mathcal{O}_P(1). $$
The asymptotics of $J'(\log(n))$ is the same as that of $|\Pi_n(\mathcal{C}^-(i,C_i))|$ in (\ref{eqn:pi.asymp}). Thus Proposition \ref{thm:ppp} follows by the Continuous Mapping Theorem.
\end{proof}

\section{Connections between the discrete and continuous setup}\label{sec:connections}
For general distribution $F$, we present a coupling argument in Proposition
\ref{lem:approx} to compare Theorem \ref{thm:main.rn} and Proposition \ref{thm:ppp}. When $F = exponential(1)$, $G = \lambda$, the coupling can be described even more explicitly, and in fact, it results in an equality in distribution. 

\begin{lem}\label{lem:connections} For $i = 0, 1, \ldots, n$, let $U_i$ be i.i.d $Uniform(0,1)$, independent of the $C_i$'s. Then
\begin{equation}\label{eqn:uici}
\Pi_n(\mathcal{C}(U_i,C_i)) \stackrel{d}{=} \Pi_n(\mathcal{C}(i,C_i)),
\end{equation}
and conditioned on $\Phi_n$ having $n+1$ points in $(0,1)^2$,
\begin{equation}\label{eqn:uiphi}
\Pi(\mathcal{C}(U_i,C_i)) \stackrel{d}{=} \Pi(\mathcal{C}(\Phi_n)).
\end{equation}
\end{lem}
\begin{proof}
Equation (\ref{eqn:uici}) follows from repeating the proof of Proposition \ref{prop:key} for the points $(U_i,C_i)$ instead of $(i,C_i)$. For the second statement, note that conditioned on $\Phi_n$ having $n+1$ points in $(0,1)^2$, these points are distributed as $(U_i,\frac{C_i}{\sum_iC_i+\epsilon})$ for some independent $\epsilon \stackrel{d}{=} exponential(1)$. In other words, these points is a version of the points $(U_i,C_i)$, rescaled by some random amount in the $y$-axis. But such rescaling does not change the partition of $[0,1]$, thus we have (\ref{eqn:uiphi}).
\end{proof}

\subsection*{Lower convex hull of points in a triangle}
Buchta \cite{buchta} considered $|\mathcal{C}(\Delta)|$, the number of faces in the lower convex hull of $(0,1)$, $(1,0)$ and $n$ points distributed uniformly at random in the triangle $\Delta$ with vertices $(0,1)$, $(0,0)$ and $(1,0)$. Theorem 1 in \cite{buchta} derives a recurrence relation for $\P(|\mathcal{C}(\Delta)| = k)$, which is exactly our $p_k^n$ in 
(\ref{eqn:pkn}). In the light of Lemmas \ref{lem:key} and \ref{lem:connections}, this connection is clear.

\begin{cor}
The partition of $[0,1]$ induced by projecting the lower faces of $\mathcal{C}(\Delta)$ onto the $x$-axis is precisely the partition of $[0,1]$ obtained by projecting faces of $\mathcal{C}^+(\Phi_n)$ onto $[0,y(V^\downarrow_0)]$ in the $y$-axis, and rescaled by $y(V^\downarrow_0)$.
\end{cor}

 Lemma \ref{lem:connections} shows that Groeneboom's result, stated in the form of (\ref{eqn:general.r}), follows directly from the asymptotics of the $Beta(2,1)$ stick-breaking in Section \ref{sec:stickbreak}. This was the spirit of Buchta's approach \cite{buchta, buchta2}. In particular, he derived $\E(|\mathcal{C}(\Delta)|)$ and $\Var(|\mathcal{C}(\Delta)|)$ exactly for finite $n$ using (\ref{eqn:pkn}). He then generalized this approach to derive the exact analogues for results in Groeneboom \cite{groenboom88}, including distribution of the number of vertices and the area outside the convex hull of a uniform sample from the interior of a convex polygon with $r$ vertices. 

\section{Lower Convex Hull of an Inhomogeneous Poisson Point Process}
\label{sec:main}
We now prove Proposition \ref{thm:ppp} in the general case. Our strategy is to generalize the proof in the previous section. Consider the vertices $(V_i^\downarrow, i \geq 0)$ of $\mathcal{C}^+(\Phi_n)$ ordered in decreasing $y$-coordinate. The analogue of Lemma \ref{lem:key} is the following.

\begin{lem}\label{lem:key.general}
For each $i \geq 1$, let $B_i = \frac{y(V^\downarrow_i)}{y(V^\downarrow_{i-1})}$.
For $s > 0$, define the distribution $I_s$ on $[0,1]$ via its cdf (also denoted $I_s$):
$$ I_s(b) := \P(B_i \leq b \mid y(V^\downarrow_{i-1}= s). $$
Then
\begin{equation}\label{eqn:is}
I_s(b) = \frac{(1-b)G([0,sb]) + \int_0^bG([0,st])dt}{\int_0^1G([0,st]) dt}, \quad  0\le b\le 1.
\end{equation}
In particular, as $s \to 0$, $I_s \to I_0 := Beta(a,2)$.
\end{lem}
\begin{proof}
Consider Figure \ref{fig:small.triangle} in Lemma \ref{lem:key}. Conditioned on $y(V^\downarrow_{i-1})=s$ and on the fact that the shaded triangle contains a point, the probability that $y(V^\downarrow_i)$ belongs to $dy$ at $y=sb$ is proportional to the area of the thin boldface rectangle of the figure w.r.t. the $\lambda\times G$ measure, that is proportional to $(1-b)d G([0,sb])$. This in turm implies the above formula for $I_s$, the cumulative distribution function of $B_i$ given $y(V_{i-1} = s)$. 
Now, as $s \to 0$, $G([0,sb]) \to C(sb)^a$ uniformly over all $b \in [0,1]$. Thus $\frac d {db}G([0,sb]) \to Cas^{a}b^{a-1}$. The term $s^{a}$ cancels, leave $I_s(db) \to I_0(db)$ at all points $b \in [0,1]$, where $I_0(db) \propto b^{a-1}(1-b)$. Thus, $I_0 = Beta(a,2)$. Since $I_0$ is continuous everywhere, $I_s \to I_0$ as distributions. (See, for example, \cite[\S 2]{durrett}).
\end{proof}

Define $S_0 = -\log(y(V^\downarrow_0))$, $S_i := -\log(y(V^\downarrow_i) = -\log(y(V^\downarrow_0)) + \sum_{j=1}^i-\log(B_i)$ for $i = 1, 2, \ldots$. Then $(S_i, i = 0, 1, \ldots)$ is a state-dependent random walk. At the $i$-th step, conditioned on $S_i = t$, $S_{i+1} - S_i$ is an independent random variable with distribution $-\log(I_{\exp(-t)})$ for the family of distributions $I_{(\cdot)}$ in Lemma \ref{lem:key.general}.
Alternatively, $(S_i)$ can be viewed as a Markov renewal process, or a Markov modulated random walk. In these settings, the inter-arrival times (or the jumps of the walk) are driven by a Markov process, in this case, $(S_i)$ itself. Since the jumps are a.s. positive, the Markov chain is transient. To the best of our knowledge, Korshunov \cite{korshunov.key, korshunov.clt} is one of the only authors who considered transient Markov renewal processes. We restate his relevant results \cite[Theorem 5]{korshunov.clt} for our case, omitting conditions which are automatically satisfied.

\begin{thm}[Korshunov's Central Limit Theorem for $S_i$ \cite{korshunov.clt}]
For $s \geq 0$, let $\xi_s$ be a random variable with distribution $I_s$. Suppose for some $\delta \in (0,1)$, $\{(\log(\xi_s))^2, 0 < s < \delta\}$ is uniformly integrable. Let $\mu = \E(-\log(\xi_0))$, $\sigma^2 = \Var(-\log(\xi_0))$. Suppose
$$ \E(-\log(\xi_s)) = \mu + o\left(\frac{1}{\sqrt{-\log(s)}}\right), $$
and
$$ \Var(-\log(\xi_s)) \to \sigma^2 $$
as $s \to 0$. Then as $s \to 0$, $t = -\log(s) \to \infty$, and
\begin{equation}\label{eqn:conclusion}
S_t/t \stackrel{a.s.}{\to} \mu, \hspace{0.5em} \mbox{ and } \hspace{0.5em} \frac{S_t - t\mu}{\sqrt{t\sigma^2}} \stackrel{d}{\to} \N(0,1).
\end{equation}
\end{thm}

\begin{lem} The assumptions of Korshunov's Central Limit Theorem are satisfied. In particular, suppose $(t_n)$ is an increasing sequence of indices, such that $t_n \to \infty$ as $n \to \infty$. Define $s_n = e^{-t_n}$. Then as $n \to \infty$, $s_n \to 0$, and (\ref{eqn:conclusion}) holds for the sequence $(t_n)$.
\end{lem}
\begin{proof} By the assumptions of Proposition \ref{thm:ppp}, as $s \to 0$,
$G([0,s]) = Cs^a + \epsilon(s)$, where $\epsilon(s) = o(s^a)$. Thus there exists some small $S$ such that for all $s < S$, for all $b \in [0,1]$, $|\epsilon(sb)| < 2|\epsilon(s)|$. Bound $I_s(b)$ in (\ref{eqn:is}) by
$$
\frac{1-2|\epsilon(s)|}{1+2|\epsilon(s)|} I_0(b) \leq I_s(b) \leq \frac{1+2|\epsilon(s)|}{1-2|\epsilon(s)|} I_0(b),
$$
where $I_0(b) = (a+1)b^a - ab^{a+1}.$ For some small $S'$, for all $s < S'$, $|\epsilon(s)|< 1/4$. So
\begin{equation}
\label{eqn:bound.fb}
|I_0(b) - I_s(b)| \leq 8|\epsilon(s)|,
\end{equation}
where this bound holds for all $s < \min(S, S')$. Note that $-\log(\xi_0)$ is a.s. positive and has finite third moments for all $a > 0$. For any $r > 0$, write
$$ \E((-\log(\xi_s))^r) = \int_0^\infty \P((-\log(\xi_s))^r > x)\, dx = \int_0^\infty\P(\xi_s > \exp(-x^{1/r}))\, dx.$$
Thus for all $s < \min(S, S')$, (\ref{eqn:bound.fb}) implies $|\E(-\log(\xi_s)^3) - \E(-\log(\xi_0)^3)| < (1+8\epsilon(s)),$ so the family of squared jumps $\{\log(\xi_s)^2, s < \min(S, S')\}$ is uniformly integrable. Similarly, we have convergence of the expectation and variance. The error in the expectation is bounded by
$$ |\E(-\log(\xi_s)) - \mu| < 8|\epsilon(s)|.$$
For any $a > 0$, $s^a = o\left(\frac{1}{\sqrt{-\log(s)}}\right)$. Since $\epsilon(s) = o(s^a)$ by assumption, $\epsilon(s) = o\left(\frac{1}{\sqrt{-\log(s)}}\right)$. 
\end{proof}

\begin{prop}\label{prop:nt}
Let $J(t)$ be the number of jumps in $[0,t]$ of the random walk $(S_i)$. Then as $t \to \infty$,
\begin{equation}\label{eqn:nt}
\frac{J(t)}{t} \stackrel{a.s.}{\to} \frac{1}{\mu}, \hspace{0.5em} \mbox{ and } \hspace{0.5em} \frac{J(t) - \mu^{-1}t}{\sqrt{\sigma^2\mu^{-3}t}} \stackrel{d}{\to} \N(0,1),
\end{equation}
where $\mu = \ds\frac{1}{a} + \frac{1}{a+1}$, $\sigma^2 = \ds\frac{1}{a^2} + \frac{1}{(a+1)^2}$.
\end{prop}
\begin{proof} This result follows from (\ref{eqn:conclusion}) by standard techniques of renewal theory. Our treatment follows that of Gut \cite[\S 2.5]{gut}. Let $\nu(t)$ be the first time in which $S_i > t$. Then $\nu(t) = J(t) + 1$ a.s. By definition,
$$ S_{J(t)} \leq t < S_{\nu(t)}.$$
Therefore,
$$ \frac{S_{J(t)}}{J(t)} \leq \frac{t}{J(t)} < \frac{S_{\nu(t)}}{\nu(t)}\cdot\frac{J(t)+1}{J(t)}. $$
Since the walk has a.s. positive increments, $J(t) \to \infty$ as $t \to \infty$ a.s. 
But $S_t/t \stackrel{a.s.}{\to} \mu$ by (\ref{eqn:conclusion}), therefore, by \cite[Theorem 2.1]{gut},
$$\frac{S_{J(t)}}{J(t)} \stackrel{a.s.}{\to} \mu, \hspace{0.5em} \mbox{ and } \hspace{0.5em} \frac{S_{\nu(t)}}{\nu(t)}\cdot\frac{J(t)+1}{J(t)} \stackrel{a.s.}{\to} \mu.$$ 
Thus $\frac{t}{J(t)} \stackrel{a.s.}{\to} \mu$, and $\frac{J(t)}{t} \stackrel{a.s.}{\to} \mu$ by the Continuous Mapping Theorem. 
Similarly,
$$ \frac{S_{J(t)} - J(t)\mu}{\sigma\sqrt{t/\mu}} \leq \frac{t - J(t)\mu}{\sigma\sqrt{t/\mu}} \leq \frac{S_{\nu(t)} - \nu(t)\mu}{\sigma\sqrt{t/\mu}} + \frac{\mu}{\sigma}\sqrt{\frac{\mu}{t}}.$$
As argued above, $J(t) \to \frac{t}{\mu}$ a.s., and therefore so does $\nu(t)$. By Anscombe's theorem \cite[\S 1.3]{gut}, \cite{anscombe}, the sequences $(S_{J(t)})$ and $(S_{\nu(t)})$ satisfy the same central limit theorem as that of $S_t$. %(We are using the original version of Anscombe's theorem, which has Anscombe's condition of uniform continuity in probability of $(S_i)$, see \cite[\S 1.3]{gut}.) 
Therefore, as $t \to \infty$,
$$\frac{t - J(t)\mu}{\sigma\sqrt{t/\mu}} = -\frac{J(t) - \mu^{-1}t}{\sqrt{\sigma^2\mu^{-3}t}} \stackrel{d}{\to} N(0,1). $$
 %In our case, choose $t = \frac{1}{a}\log(n)$ to achieve the same conclusion as $n \to \infty$. %As argued previously, $J(\frac{1}{a}\log(n)) - M_n \stackrel{P}{to} 0$. By continuous mapping theorem, $M_n$ has the same limiting distribution as $J(\frac{1}{a}\log(n))$ above.
\end{proof}

Recall that $|\mathcal{C}^+(\Phi_n)| = J[-\log(y(V^\uparrow_0)) + \log(y(V^\downarrow_0))]$. Since $F(x) \sim Cx^a$ for $x$ near $0$, $y(V^\uparrow_0) \sim F^{-1}(\frac{1}{n}) \sim \frac{1}{C}n^{-1/a}$, so we expect $-\log(y(V^\uparrow_0)) \sim \frac{1}{a}\log(n)$ and $|\mathcal{C}^+(\Phi_n)| \approx J(\frac{1}{a}\log(n))$. Formally, we have the following analogue of Lemma \ref{lem:approx}. 

\begin{lem}\label{lem:approx.general}
Let $\delta(n) = Z_n^+ - J(\frac{1}{a}\log(n))$. Then $\delta(n) = \mathcal{O}_P(1)$.
\end{lem}
\begin{proof}
The proof is along the same lines as Lemma \ref{lem:approx}. By (\ref{eqn:nt}), it is sufficient to show that for a function $\ell$ tending to infinity arbitrarily slowly, with high probability,
\begin{equation}\label{eqn:bound.ymn.general}
\frac{1}{a}\log(n)-\ell(n) \leq -\log(y(V^\uparrow_0)) \leq \frac{1}{a}\log(n)+\ell(n).
\end{equation}
%and
%\begin{equation}\label{eqn:bound.y0.general}
%-\ell(n) \leq \log(y(V_0)) \leq 0.
%\end{equation}
Again, we divide the region $(0,1) \times (0,\infty)$ into $n$ vertical strips $(i/n,(i+1)/n) \times (0,\infty)$ for $i = 0, \ldots, n-1$. In each strip, the $y$-coordinates of the points of $\Phi_n$ form an inhomogeneous Poisson point process on $(0,\infty)$ with intensity measure $G$. In particular, for the $i$-th strip, the first $y$ jump $\epsilon_i$ has distribution $F$. Define $Y' := \min(\epsilon_1, \ldots, \epsilon_n)$. Then $Y' \stackrel{d}{=} F^{-1}(B)$ where $B \stackrel{d}{=} Beta(1,n)$. Now, 
$$ \P(B \leq \frac{\ell(n)}{n}) = 1 - \left(1-\frac{\ell(n)}{n}\right)^n \to 1 $$
as $n \to \infty$. Conditioned on the event $B \leq \frac{\ell(n)}{n}$, for large $n$,
$$ F^{-1}(B) \leq \frac{2}{C}B^{1/a} \leq \frac{2}{C}\frac{\ell(n)^{1/a}}{n^{1/a}}.$$
Therefore, for large $n$,
$$ \P\{-\log(Y') \geq \frac{1}{a}\log(n) - \frac{1}{a}\log(\ell(n)) - \log\left(\frac 2C\right)\} = 1 - \left(1-\frac{\ell(n)}{n}\right)^n. $$
Similarly, 
$$ \P\left(B \leq \frac{1}{n\ell(n)}\right) = 1 - \left(1-\frac{1}{n\ell(n)}\right)^n \to 0 $$
as $n \to \infty$. By the same argument, for large $n$,
$$ \P\{-\log(Y') \leq \frac{1}{a}\log(n) + \frac{1}{a}\log(\ell(n)) + \log\left(\frac 2 C\right) \} = 1 - \left(1-\frac{1}{\ell(n)n}\right)^n. $$
Therefore, (\ref{eqn:bound.ymn.general}) holds with high probability for $Y'$ in place of $y(V^\uparrow_0)$. But $Y' = y(V^\uparrow_0)$ conditioned on the event $Y' < G^{-1}(1)$, with $G^{-1}(1)$ the real number such that $G([0,G^{-1}(1)]=1$, namely $G^{-1}(1)= F^{-1} (1-e^{-1})$. This is precisely the event that there is at least one point of $\Phi_n$ in the square $(0,1) \times (0,G^{-1}(1))$, so this event happens with probability $1 - e^{-n}$. Thus, (\ref{eqn:bound.ymn.general}) happens with probability at least 
$$ (1 - e^{-n})\left(1 - \left(1-\frac{1}{\ell(n)n}\right)^n\right) \to 1. $$

%The event in (\ref{eqn:bound.y0.general}) is the event that $y(V_0) \geq \frac{1}{\ell(n)}$. Let $Y \stackrel{d}{=} F$. Then $y(V_0)$ is distributed as $Y$ given $Y\leq G^{-1}(1)$. So
%$$ \P(y(V_0) > \frac{1}{\ell(n)}) = \frac{\P(Y >\frac{1}{\ell(n)}) - \P(Y>G^{-1}(1))}{1-\P(Y>G^{-1}(1))} = \frac{e^{-G(\frac{1}{\ell(n)})} - e^{-1}}{1-e^{-1}} \to 1 $$
%as $n \to \infty$, since $G(\frac{1}{\ell(n)}) \to 0$. So (\ref{eqn:bound.y0.general}) holds.
%
%Finally, for large $n$, distribution of $y(V_0)$ and $y(V_{Z_n^+})$ are almost independent. Thus the probability of the intersection of events (\ref{eqn:bound.ymn.general}) and (\ref{eqn:bound.y0.general}) is approximately their product, which tends to $1$ as $n \to \infty$.
\end{proof}
\begin{proof}[Proof of Proposition \ref{thm:ppp}]
The argument is exactly the same as in the case where $F$ is exponential given in Section \ref{sec:simple.poisson}. In summary, Proposition \ref{prop:nt} establishes the central limit theorem for $J(\frac{1}{a}\log(n))$. The number of faces to the left and right of $V_0^\uparrow$, $|\mathcal{C}^+(\Phi_n)|$ and $|\mathcal{C}^-(\Phi_n)|$, are independent conditioned on $y(V_0^\uparrow)$. Lemma \ref{lem:approx.general} states $|\mathcal{C}^+(\Phi_n)|$ and $|\mathcal{C}^-(\Phi_n)|$ are simultaneously well-approximated by two independent copies of $J(\frac{1}{a}\log(n))$. Thus $N(\Phi_n)$ is distributed like their sum, and this concludes the proof. 
\end{proof}

For $\alpha \in (0,\infty)$, recall that $L(\alpha)$ is the line orthogonal to $(1,\alpha)$ which supports $\mathcal{C}^+(\Pi_n)$, and that $L_x(\alpha)$ is its $x$-intercept. Note that $L_x(\cdot)$ is a pure jump process indexed by $\alpha$. Let $L_x^i, i = 1, \ldots, |\mathcal{C}^+(\Phi)|$ be the sequence of values of $L_x(\cdot)$ ordered in increasing value. Define $L_x^0 = x(V^\downarrow_0)$. Consider the triangles by the consecutive lines and the $x$-axis as in Figure \ref{fig:triangle} below. For $i = 1, 2, \ldots, |\mathcal{C}^+(\Phi)|$, let $T_i$ be the $i$-th triangle, $area(T_i)$ denote its area with respect to the measure $\lambda \times G$. 

\begin{center}
\begin{figure}[h]
\includegraphics[width=0.8\textwidth]{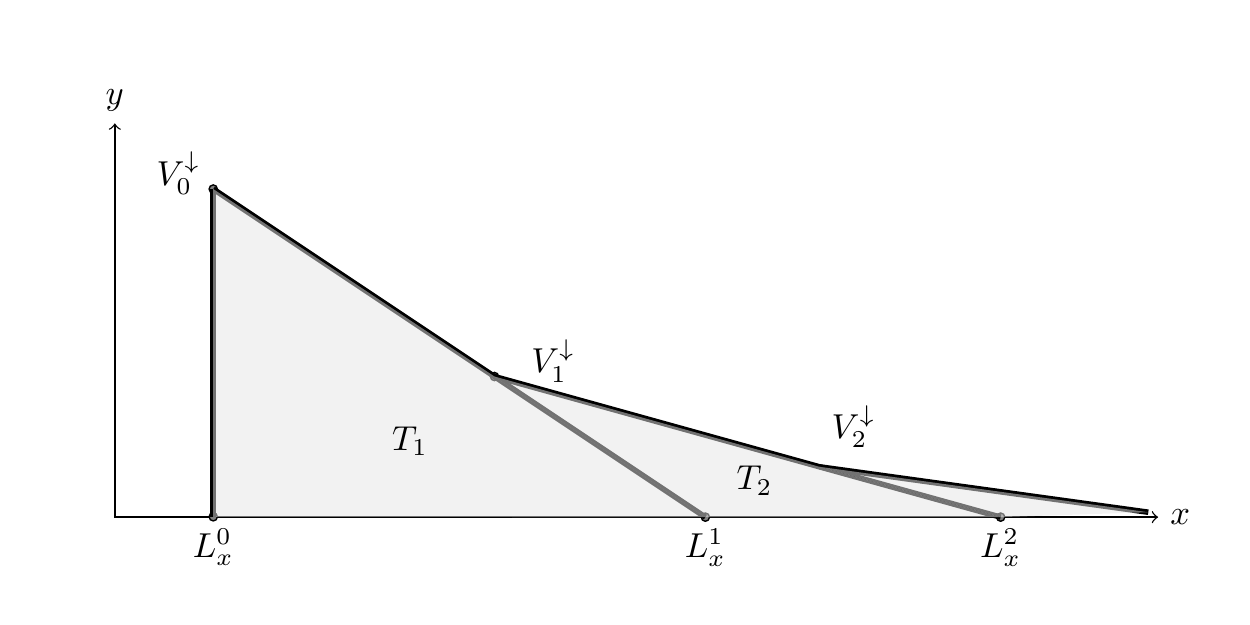}
\caption{Divide up the region between $\mathcal{C}^+(\Phi_n)$ and the $x$-axis into triangles based on the jumps of $L_x(\cdot)$. \protect \label{fig:triangle}}
\end{figure}
\end{center}

\begin{cor}
For $i = 1, 2, \ldots, R^+_n$, $area(T_i)$ are i.i.d. $exponential(n)$, independent of the vertices $V^\downarrow_i$ of $\mathcal{C}^+(\Phi_n)$.
\end{cor}
\begin{proof}
Let $S = L_x^i - L_x^{i-1}$ be the $i$-th increment of the process $L_x(\cdot)$. Conditioned on $V^\downarrow_{i-1}$ and $L_x^{i-1}$, let $T_S$ denote the triangle with vertices $V^\downarrow_{i-1}$, $(L_x^{i-1},0)$ and $(L_x^{i-1} + S,0)$. Now, 
$$area(T_S) = S\int_0^{y(V^\downarrow_0)}(1-y/y(V^\downarrow_0)) dG([0,y]) = S f(y(V^\downarrow_0)), $$
where $f(y(V^\downarrow_0))$ is the above integral, which is independent of $S$. We have
$$ \P(S > s) = \exp(-n\ (area(T_s))) = \exp(-nsf(y(V^\downarrow_0))). $$
This implies
$$ \P(area(T_S) > t) = \P\left(S > \frac{t}{f(y(V^\downarrow_0))}\right) =
\exp\left(-n\frac{t}{f(y(V^\downarrow_0))}f(y(V^\downarrow_0))\right) = \exp(-nt). $$
Therefore, $area(T_S)$, independent of $V^\downarrow_{i-1}$, $L_x^{i-1}$ and $S$, is distributed as an exponential random variable with rate $n$. 

\end{proof}
Groeneboom \cite{groenboom12} proved this when $G$ is the Lebesgue measure. He used it to derive the asymptotics for the sum $A_n = \sum_i area(T_i)$ jointly with $|\mathcal{C}^+(\Phi_n)|$. Since we approximate $|\mathcal{C}^+(\Phi_n)|$ by the number of renewals in a fixed interval (cf. Lemma \ref{lem:approx.general}), the asymptotic normality for $A_n$ easily follows. For general $G$, let us compute the expectation and variance of $A_n$ for large $n$. We have
\begin{align*}
\E(A_n) &= \E(\E(A_n|J(\log(n)/a))) = \E(J(\log(n)/a)) = \frac{a+1}{2a+1}\log(n), \\
\Var(A_n) &= \E(J(\log(n)/a)) + \Var(J(\log(n)/a)) = \frac{6 a^3+8 a^2+4 a+1}{(2 a+1)^3}\log(n).
\end{align*}
For $a = 1$, this reduces to $\E(A_n) = \ds\frac{2}{3}\log(n)$, $\Var(A_n) = \ds\frac{28}{27}\log(n)$, as showed in \cite{groenboom12, buchta2}. See \cite{schneider} for a historical review and summary of recent developments on asymptotics of $A_n$ in higher dimensions.

Finally, we note that Proposition \ref{thm:ppp} is stated with the Poisson point process being restricted to the rectangle $(0,1) \times (0,G^{-1}(1))$. If we widen this rectangle to $(0,1) \times (0,2G^{-1}(1))$, the probability of points in  $(0,1) \times (G^{-1}(1),2G^{-1}(1))$ being a vertex of the lower convex hull is clearly very small. Thus Proposition \ref{thm:ppp} also holds for the lower convex hull of points from the infinite strip $(0,1) \times (0,\infty)$. We chose to state it for the rectangle to make the role of $n$ clear: on the rectangle  $(0,1) \times (0,G^{-1}(1))$, we have $Poisson(n)$ points. This makes conditioning arguments such as that in Lemma \ref{lem:connections} a little more convenient.

\section{Proof of the main theorem}\label{sec:coupling}

The discrete case corresponds to a PPP  on $(0,1) \times (0,\infty)$ with intensity measure $\lambda_n \times G$, where $\lambda_n$ is a discrete measure on $\R$ which puts mass $1$ at every point $i/n$ for $i = 0, 1, \ldots$, and $0$ elsewhere. Clearly $\frac{1}{n}\lambda_n \to \lambda$ in the space of measures, thus we expect the discrete and continuous cases to have the same asymptotics. We make this rigorous through a direct coupling. 

Divide $(0,1) \times (0,\infty)$ into $n$ vertical strips $(i/n,(i+1)/n) \times (0,\infty)$, $i = 0, \ldots, n-1$. Form the new point process $\widetilde{\Phi}_n$ from $\Phi_n$ as follows: for each point $(X_i,Y_i)$ in the $i$-th strip in $\Phi_n$, place a point $(i/n,Y_i)$ in~$\widetilde{\Phi}_n$. This produces an a.s. bijection $\psi: \Phi_n \to \widetilde{\Phi}_n$, such that a point $(X_i,Y_i)$ of $\Phi_n$ and its image $\psi((X_i,Y_i))$ have equal $y$-coordinates, and differ by at most $\frac{1}{2n}$ in their $x$-coordinates. We use this coupling to show the following, which implies that Theorem~\ref{thm:main.rn} is equivalent to Proposition \ref{thm:ppp}. 

\begin{prop}\label{prop:coupling}
$$ |\mathcal{C}^+(\Phi_n)| - |\mathcal{C}^+(\widetilde{\Phi}_n)| = \mathcal{O}_P(1). $$
\end{prop}
\begin{proof}
Recall that $L(\alpha)$ is the line supporting $\mathcal{C}^+(\Phi_n)$ with slope $\alpha$. Define
$$ L^{1/n}(\alpha) = \{(x,y) \in \R^2: \exists \epsilon \in (0,1/n] \mbox{ such that } (x-\epsilon,y) \in L(\alpha)\}.$$
Let $(\tilde{x}(\alpha),\tilde{y}(\alpha))$ be the vertex supported by the vector $(1,\alpha)$ in $\widetilde{\Phi}_n$. Then almost surely, $(\tilde{x}(\alpha),\tilde{y}(\alpha)) = \psi(V)$ for some point $V$ of $\Phi_n$ lying in $L^{1/n}(\alpha)$ (see Figure \ref{fig:ln}). 
Let us condition on the vertices of $\mathcal{C}^+(\Phi_n)$ and the values $L_x^i$, $i = 0, 1, \ldots, |\mathcal{C}^+(\Phi_n)|-1$. Define $Y^\ast = \min\{L_y(\alpha_0), G^{-1}(1)\}$. Extend $\Phi_n$ to include the point $(Y^\ast,0)$, and consider the lower convex hull, as in Figure \ref{fig:intersection}. Let $P^{1/n}$ denote the collection of points not in this convex hull, but whose $x$-coordinate at most $\frac{1}{2n}$ away from this convex hull. That is,
$$ P^{1/n} = \{(x,y) \in \R^2: \exists \epsilon \in (0,1/n] \mbox{ such that } (x-\epsilon,y) \in \mathcal{C}^+(\Phi_n \cup (0,Y^\ast))\}. $$ 
By definition, $\Phi_n$ has no point below $\mathcal{C}^+(\Phi_n \cup (0,Y^\ast))$. Thus $|\mathcal{C}^+(\widetilde \Phi_n)|$ is at most the number of points of $\Phi_n$ in $P^{1/n} \cup \, \mathcal{C}^+(\Phi_n \cup (0,Y^\ast))$.

\begin{center}
\begin{figure}[h]
\includegraphics[width=0.6\textwidth]{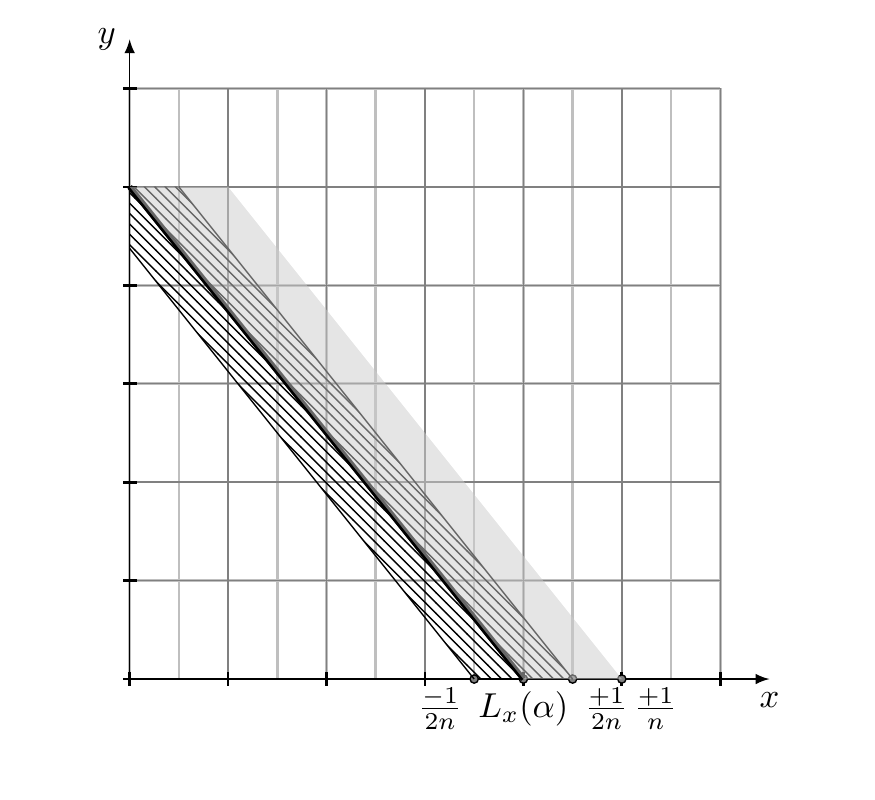}
\caption{The thick line is $L(\alpha)$, the gray region is $L^{1/n}(\alpha)$. The vertex $(\tilde{x}(\alpha),\tilde{y}(\alpha))$ of $\widetilde{\Phi}_n$ has to lie in the stripped region. Thus $(\tilde{x}(\alpha),\tilde{y}(\alpha)) = \psi(V)$ for some point $V$ of $\Phi_n$ in $L^{1/n}(\alpha)$. \protect \label{fig:ln} }
\end{figure}
\end{center}

\begin{center}
\begin{figure}[h]
\includegraphics[width=0.8\textwidth]{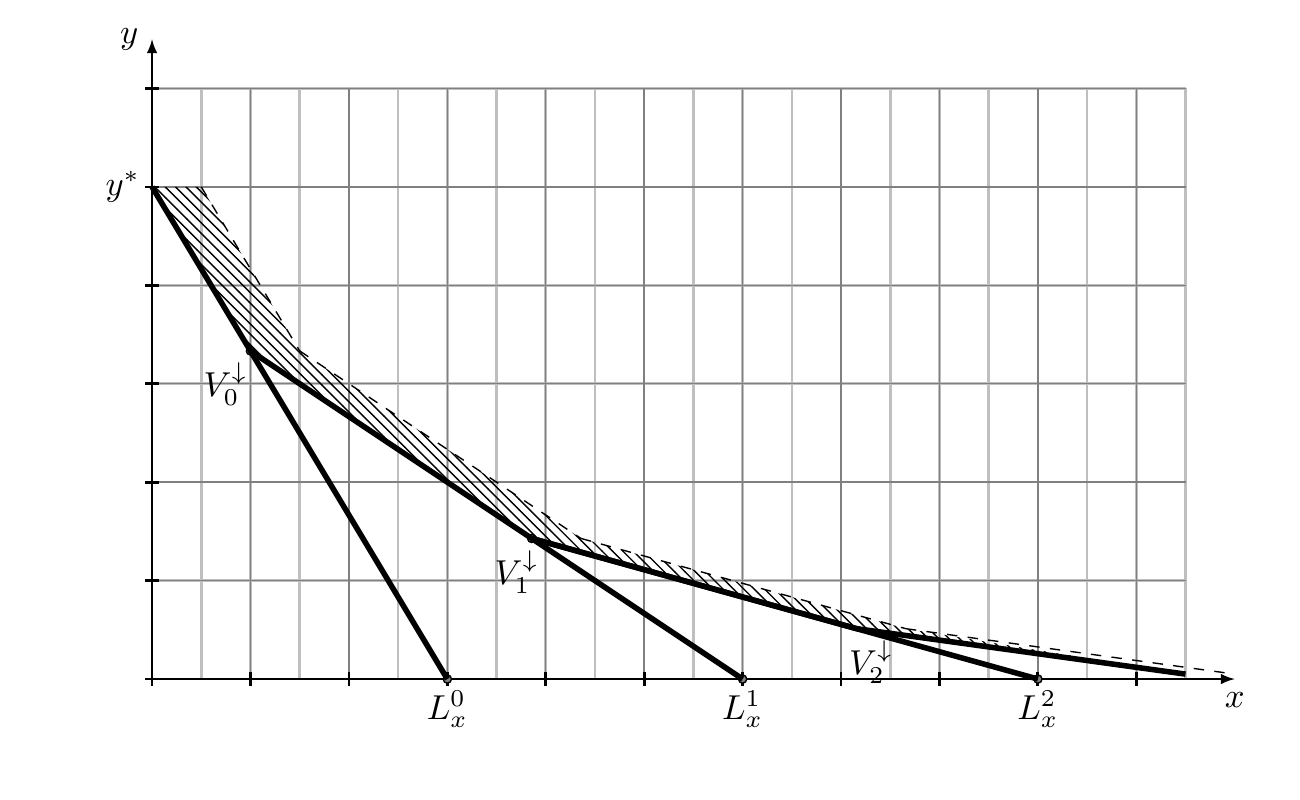}
\vskip-1.2cm
\caption{Conditioned on the vertices of $\mathcal{C}^+(\Phi_n)$, $Y^\ast$ and $L_x^i$, $|\mathcal{C}^+(\widetilde{\Phi}_n)|$ is at most the number of points of $\Phi_n$ lying within $x$-distance $1/n$ of the lower convex hull. This is precisely the region shaded, which consists of parallelograms of width $1/n$. \protect \label{fig:intersection}}
\end{figure}
\end{center}

Divide the region $P^{1/n}$ into parallelograms of width $1/n$. The vertices of the $i$-th parallelogram are $V_i^\downarrow, (x(V_i^\downarrow)+1/n, y(V_i^\downarrow)), V_{i+1}^\downarrow$ and $(x(V_{i+1}^\downarrow)+1/n, y(V_{i+1}^\downarrow))$. Conditioned on $y(V_i^\downarrow) = y_i$, $Y^\ast = y^\ast$, and conditioned on $|\mathcal{C}^+(\Phi_n)| = r$, the area of $P^{1/n}$ under $\lambda \times G$~is
$$ G(y^\ast) - G(y_1) + \sum_{i=1}^{r-1}(G(y_i) - G(y_{i+1})) = G(y^\ast) - G(y_r) \leq 2. $$
since $y^\ast, y_r \leq G^{-1}(1)$. By definition, $\Phi_n$ has $r$ vertices in $\mathcal{C}^+(\Phi_n)$. Therefore,
$$ |\mathcal{C}^+(\widetilde{\Phi}_n) - r| \leq Poisson(2) = \mathcal{O}_P(1) $$
for all realizations of the point process $\Phi_n$.
\end{proof}

\section{Discussion}\label{sec:discuss}

We considered random tropical polynomials  $\Tf_n(x) = \min_{i=1,\ldots,n}(C_i + ix)$ where the coefficients $C_i$ are i.i.d. random variables with some c.d.f. $F$ with support on $(0,\infty)$. We showed that $Z_n$, the number of zeros of $\Tf_n$ satisfies a central limit theorem under mild assumptions on the rate of decay of $F$ near $0$. Specifically, if $F$ near $0$ behaves like the $gamma(a,1)$ distribution for some $a > 0$, then $Z_n$ has the same asymptotics as the number of points on the interval $[0,\log(n)/a]$ of a renewal process with inter-arrival distribution $-\log(Beta(2,a))$. 
The proof techniques draw on connections between random partitions, renewal theory and random polytopes constructed from Poisson point processes. They lead to simpler proofs of the central limit theorem for the number of vertices of the convex hull of $n$ uniform random points in a square.

The assumption that the support of $F$ is $[0,\infty)$ can easily be extended to the case with support on $(c,\infty)$ for some constant $c$, provided the behavior of $F$ near $c$ is as above. This follows from the fact that the number of vertices of a polytope is invariant under translation and scaling by constants. It is crucial, however, that $F$ be a continuous distribution. In particular, Theorem \ref{thm:main.rn} does not hold for discrete distributions. Indeed, if $F(0) = p > 0$, $Z_n$ is at most the sum of two independent $Geometric(p)$ random variables for all $n$, and certainly does not have a normal scaling.

This work is a first stab at \emph{stochastic tropical geometry}, the study of linear functionals and intersections of random tropical varieties. These are common zeros of a collection of random tropical polynomials. In fields with valuations, they are precisely the tropicalization of random algebraic varieties. By considering these varieties at random, we gain insights into the global structure of tropical varieties and their preimages as a collection of sets. Unlike classical varieties, the tropical analogues are polyhedral in nature. Random tropical varieties are strongly connected with random polytopes, a rich branch of stochastic geometry \cite{SKM,schneider,SW}. This is a key ingredient in our proof of Theorem \ref{thm:main.rn}. 

Our next steps will focus on random tropical polynomials in several variables and system of random tropical polynomials. A tropical polynomial in $m$ variable is a map $\Tf: \R^m \to \R$, given by
$$ \Tf(x) = \min_{i \in I}(C_i + i\cdot x), $$
where $C_i \in \R$, $I \subset \mathbb{Z}^m$ is some indexing set, and $\cdot$ is the usual inner product in $\R^m$. The convex hull of $I$ is called the Newton polytope, and its subdivision by $\mathcal{C}(i,C_i)$ is a regular subdivision, or in other words, weighted Delaunay triangulations \cite[\S 2.3]{bernd.trop}. Thus, random tropical polynomials generate a type of random partition of subsets of $\mathbb{Z}^m$. It would be very interesting to understand this lattice partition. For example, if we consider a random tropical polynomial of degree $n$ in $m$ variables with i.i.d coefficients $C_i$, as $n \to \infty$, is there a scaling limit for the number of cells of such partitions? 

\bibliographystyle{plain}
\bibliography{references}
%TODO:
%consistent notation: (0,1) vs [0,1], O(1) vs \O(1), capital vs lower case , log(n)/a
%be clear about when variables are being conditioned, etc... 
\end{document}